\documentclass[12pt]{amsart}
\usepackage{latexsym, url}
\usepackage{amsthm}
\usepackage{amsmath}
\usepackage{amsfonts}
\usepackage{amssymb}
\usepackage[dvips]{graphicx}
\usepackage{xypic}
\input xy
\xyoption{all}

\newtheorem{theo}{Theorem}[section]
\newtheorem{lemma}[theo]{Lemma}

\newtheorem{propo}[theo]{Proposition}
\newtheorem{defi}[theo]{Definition}
\newtheorem{coro}[theo]{Corollary}
\newtheorem{rem}[theo]{Remark}
\newtheorem{pb}[theo]{Problem}

\newtheorem{exam}[theo]{Example}

\newcommand\Inj{\operatorname{Inj}}

\newcommand\ap{\operatorname{ap}}

\newcommand\op{\operatorname{op}}
\newcommand\id{\operatorname{id}}

\newcommand\Set{\operatorname{\bf Set}}
\newcommand\Met{\operatorname{\bf Met}}
\newcommand\SMet{\operatorname{\bf SMet}}
\newcommand\Ban{\operatorname{\bf Ban}}

\newcommand\cof{\operatorname{cof}}

\newcommand\colim{\operatorname{colim}}

\newcommand\ca{\mathcal {A}}
\newcommand\eps{\mathcal {\varepsilon}}

\newcommand\cd{\mathcal {D}}
\newcommand\cf{\mathcal {F}}

\newcommand\ce{\mathcal {E}}

\newcommand\ck{\mathcal {K}}
\newcommand\cl{\mathcal {L}}

\newcommand\cs{\mathcal {S}}

\date{October 24, 2016}
 
\begin{document}
\title[Approximate injectivity]
{Approximate injectivity}
\author[J. Rosick\'y and W. Tholen]
{J. Rosick\'y$^{*}$ and W. Tholen$^{**}$}
\thanks{$^{*}$Supported by the Grant Agency of the Czech Republic under the grant 
              P201/12/G028. 
              $^{**}$Supported by the Natural Sciences and Engineering Research Council of Canada under the Discovery Grants program. } 
\address{
\newline J. Rosick\'{y}\newline
Department of Mathematics and Statistics\newline
Masaryk University, Faculty of Sciences\newline
Kotl\'{a}\v{r}sk\'{a} 2, 611 37 Brno, Czech Republic\newline
rosicky@math.muni.cz
\newline
W. Tholen\newline
Department of Mathematics and Statistics\newline
York University, Faculty of Science\newline
4700 Keele Street, Toronto, Ontario M3J 1P3, Canada\newline
tholen@mathstat.yorku.ca
}
 
\begin{abstract}
In a locally $\lambda$-presentable category, with $\lambda$ a regular cardinal, classes of objects that are injective with respect to a family of morphisms whose domains and codomains are $\lambda$-presentable, are known to be characterized by their closure under products, $\lambda$-directed colimits and $\lambda$-pure subobjects. Replacing the strict commutativity of diagrams by ``commutativity up to $\eps$", this paper provides an ``approximate version" of this characterization for categories enriched over metric spaces. It entails a detailed discussion of the needed $\eps$-generalizations of the notion of $\lambda$-purity. The categorical theory is being applied to the locally $\aleph_1$-presentable category of Banach spaces and their linear operators of norm at most 1, culminating in a largely categorical proof for the existence of the so-called Gurarii Banach space.
\end{abstract} 

\maketitle
\section{Introduction}
Recall that an object $K$ is \textit{injective} to a morphism $f:A\to B$ in a category $\ck$ if, for every morphism $g:A\to K$, there is a morphism
$h:B\to K$ with $hf=g$. There is a well-developed theory of injectivity in locally presentable categories (see \cite{AR}), playing an important role
in both algebra and topology. This theory applies to Banach spaces, too, because the category $\Ban$ of (real or complex) Banach spaces and their linear operators 
of norm at most 1 is locally $\aleph_1$-presentable (see \cite{AR}, 1.48). But in this category there is another --and probably more important-- concept, of so-called {\em approximate
injectivity} (see \cite{L}), which is based on the fact that $\Ban$ is enriched over metric spaces. The basic idea is to replace the commutativity of diagrams
by their $``$commutativity up to $\eps$". The aim of our paper is to develop a theory of approximate injectivity in metric enriched categories analogously
to that of injectivity in ordinary locally presentable categories, and apply it to $\Ban$. In particular, we present a largely categorical existence proof for the {\em Gurarii space}, which attracted renewed attention in several recent papers (see, for example, 
 \cite{ASCGM, GK, K2}).

Let us first clarify the categorical context of this paper in as concrete terms as possible. The category $\Met$ of  metric spaces and non-expansive maps is neither complete nor complete, and the tensor product $X\otimes Y$, which for $X=(X,d), Y=(Y,d')$ puts the $+$-metric $d\otimes d'((x,y),(x',y'))=d(x,x')+d'(y,y')$ on $X\times Y$, fails to make $\Met$ monoidal closed. (Note that $X\otimes Y$ must not be confused with the Cartesian product $X\times Y$ in $\Met$, which is given by the max-metric.) 
One therefore enlarges $\Met$ to the category $\Met_{\infty}$ of {\em generalized metric spaces}, by allowing distances to be $\infty$ while keeping all other requirements, including the type of morphisms. Then $\Met_{\infty}$ is complete and cocomplete and monoidal closed, with the internal hom
providing the hom-set  $\Met_{\infty}(X,Y)$ with the sup-metric 
 $d''(f,g)=\sup\{d'(fx,gx)\;|\;x\in X\}$. In what follows, we will normally denote the (generalized) metric of a space by $d$, using annotations or variations only for the sake of clarity. 
 
Throughout most of this paper, we will be considering a $\Met_{\infty}${\em -enriched category} $\ck$. Hence, $\ck$ has a class ${\rm ob}\ck$ of objects, and, for all $A,B,C\in{\rm ob}\ck$, there are hom-objects $\ck(A,B)$ in $\Met_{\infty}$ with non-expansive composition maps $\ck(B,C)\otimes\ck(A,B)\to\ck(A,C)$ and units $1\to \ck(A,A)$ (where 1 is a one-point metric space), satisfying the expected associativity and unity conditions. Interpretation of a $\Met_{\infty}$-arrow $1 \to \ck(A,B)$ as a morphism $A\to B$ defines the underlying ordinary category $\ck_0$ of $\ck$, which must be carefully distinguished from $\ck$ (see \cite{Kelly1} for details). Often we will nevertheless call a morphism in $\ck_0$ a morphism in $\ck$. Should all hom-objects $\ck(A,B)$ of the $\Met_{\infty}$-enriched category $\ck$ happen to be ordinary metric spaces, we will allow ourselves to briefly call $\ck$ a $\Met$-{\em enriched category}.
 
Our principal examples of $\Met$-enriched categories arise from concrete categories $\ck$ over $\Met$, so that one has a faithful functor $U:\ck\to\Met$. Then the hom-set $\ck(A,B)$ can be considered a subspace of $\Met_{\infty}(UA,UB)$, and the composition maps of $\ck$ remain non-expansive. 
The example of primary interest in this context is the category $\Ban$, to be considered as a concrete category over $\Met$ via the unit ball functor $U:\Ban\to\Met$,  given by  $\Ban(l_1(1),-)$,
where $l_1$ is the left adjoint to the unit ball functor $\Ban\to\Set$. 
 
We have to carefully distinguish between limits in $\ck_0$ and \textit{(conical) limits} in $\ck$, the latter being limits of the former type ({\em i.e.}, limits in $\ck_0$) that are preserved by all representables $\ck(K,-):\ck_0\to\Met_{\infty}$ (see \cite{Kelly1}, section 3.8). For concrete $\Met$-categories, these are limits preserved by 
$U:\ck\to\Met$. Likewise, \textit{conical colimits} in $\ck$ are colimits in $\ck_0$ preserved by the representables $\ck(-,K):\ck_0^{\op}\to\Met$, a property which, 
 for concrete $\Met$-categories, reduces to the preservation of colimits by $U$. 

Recall that, for a regular cardinal $\lambda$, an object $K$ of the ordinary category $\ck_0$ is $\lambda$-{\em presentable} if its representable functor $\ck_0(K,-):\ck_0\to \Set$ preserves $\lambda$-directed colimits. $\ck_0$ is {\em $\lambda$-accessible} if it has all $\lambda$-directed colimits and a set $\ca$ of $\lambda$-presentable objects such that every object in $\ck_0$ is a $\lambda$-directed colimit of objects in $\ca$; if all small colimits exist, $\ck_0$ is {\em locally $\lambda$-presentable}. Accessible (locally presentable) means $\lambda$-accessible (locally $\lambda$-presentable, respectively) for some $\lambda$.
The underlying ordinary category of $\Met_{\infty}$ is locally $\aleph_1$-presentable (see \cite{LR}, 4.5(3)), and for $\lambda$ an uncountable regular cardinal, a generalized metric space $X$ is $\lambda$-presentable if, and only if, $|X|<\lambda$. An object $K$ in a $\Met_{\infty}$-enriched category $\ck$ is 
$\lambda$-{\em presentable} (in the enriched sense) if $\ck(K,-):\ck\to\Met_{\infty}$ preserves $\lambda$-directed colimits. Again, the enriched notion must be distinguished from the ordinary notion of $K$ being $\lambda$-presentable in $\ck_0$, 
which postulates only that the $\Set$-valued functor $\ck(K,-):\ck_0\to\Set$ preserve $\lambda$-directed colimits. Since, for $\lambda$ uncountable, the forgetful functor 
$V:\Met_{\infty}\to\Set$ preserves $\lambda$-directed colimits, in this case $\lambda$-presentability of an object in $\ck$ implies its $\lambda$-presentability in $\ck_0$. But since $V$ does not create $\lambda$-directed colimits, the converse statement generally fails. However, following \cite{Kelly2} 5.5 and 7.4, it does hold for $\ck=\Met_{\infty}$, as well as in the case $\ck=\Ban$ when $\lambda$ is uncountable, as  we confirm now with the following Lemma.

\begin{lemma}\label{le1.1} 
Let $\lambda$ be an uncountable regular cardinal. Then any Banach space $\lambda$-presentable in $\Ban_0$ is $\lambda$-presentable in $\Ban$.
\end{lemma}
\begin{proof}
For an uncountable regular cardinal $\lambda$, a Banach space $A$ is $\lambda$-presentable in $\Ban_0$ if and only if it has a dense subset
of cardinality $<\lambda$. Consider a $\lambda$-directed colimit $(b_i:B_i\to B)_{i\in I}$ in $\Ban$ and a Banach space $A$ which is  $\lambda$-presentable in $\Ban_0$. 
We have to show that $\Ban(A,b_i):\Ban(A,B_i)\to\Ban(A,B)$ is a $\lambda$-directed colimit in $\Met$. Clearly, $V$ sends this cocone to a $\lambda$-directed colimit in $\Set$. Consider $f,g:A\to B$. It remains to be shown that $d(f,g)=\inf d(f_i,g_i)$, where $f=b_if_i$ and $g=b_ig_i$. Since $\lambda$ is uncountable, 
for each $a\in A$ there is $i\in I$ such that $d(fa,ga)=d(f_ia,g_ia)$. Since $A$ has a dense subset of cardinality $<\lambda$, there is $i\in I$ such that
$d(f,g)=d(f_i,g_i)$.
\end{proof}

In the following section we briefly introduce the framework of $\eps$-commutativity (= ``commutativity up to $\eps$") in $\Met_{\infty}$-enriched categories, as well as the ensuing concept of $\eps$-(co)limit. Having presented $\eps$-versions of the notion of pure subobject in Section 3, we proceed to give sufficient conditions for a class of objects in a locally presentable $\Met_{\infty}$-enriched 
category to be an $\eps$-injectivity class (Theorem \ref{th4.8}), which leads us to a full characterization of approximate injectivity (= $\eps$-injectivity, for all $\eps>0$) classes, in terms of their closure under products, directed colimits and appropriately generalized pure subobjects (Theorem \ref{th5.5}). The last section is devoted to presenting a categorical framework for constructing the Urysohn metric space and the Gurarii Banach space.

\section{$\varepsilon$-homotopy}
For any $\eps \in[0,\infty]$ and morphisms $f,g:A\to B$ in a $\Met_{\infty}$-enriched category $\ck$, we say that $f$ is $\eps${\em -homotopic} (or $\eps${\em -close} \cite{K}) to $g$, if $d(f,g)\leq\eps$ in the generalized metric $d$ of $\ck(A,B)$; we write
$$
f\sim_\eps g\Leftrightarrow d(f,g)\leq\eps.
$$
$f:A\to B$ is an $\varepsilon${\em -homotopy equivalence} if there exists $f':B\to A$ with $f'f\sim_\varepsilon\id_A$ and $ff'\sim_\varepsilon\id_B$. 
(This concept is related to the Gromov-Hausdorff distance of $A$ and $B$; see \cite{L}, 2.4.)

The relation $\sim_\varepsilon$ is preserved by composition from either side, and it is reflexive and symmetric, but generally not transitive; rather, one has
 the obvious {\em transitivity rule for $\eps$-homotopy}, which just rephrases the triangle inequality:
\[f\sim_\eps g \quad {\rm{and}} \quad g\sim_\delta h \Longrightarrow f\sim_{\eps+\delta} h.\]
$\eps${\em -commutativity} of diagrams in $\ck$ has the obvious meaning. For example, to say that
$$
\xymatrix@=4pc{
B \ar[r]^{\overline{g}} & D \\
A\ar [u]^{f} \ar [r]_{g} & C \ar[u]_{\overline{f}}
}
$$
is an $\eps$-commutative square simply means
$\overline{f}g\sim_\varepsilon\overline{g}f$.

\begin{rem}\label{re2.1}
{\em
Our motivation for using the homotopic terminology arises from the case $\ck=\Met_{\infty}$, as follows.
For $\eps>0$, let $2_\varepsilon$ be the space with two points whose distance is $\varepsilon$, and we put $2_0=1$.
Then $2_\eps$ is not finitely presentable because it is
a colimit of the chain formed by the spaces $2_{\eps+\frac{1}{n}}$.
With the injection
$$
i_\varepsilon: 1+1\to 2_\varepsilon
$$
we get the weak factorization system $(\cof(i_\varepsilon),i_\varepsilon^\square)$ (see \cite{AHRT}). Clearly, for any morphism $h$ in $\Met_{\infty}$, one has $i_\varepsilon\square h$ if, and only if,  $d(hx,hy)\leq\varepsilon$
implies that $d(x,y)\leq\varepsilon$ for all $x,y$ in the domain of $h$, and since $h$ is non-expansive, we have the converse implication too. 
Consequently, $2_\varepsilon$ is the induced cylinder object, {\em i.e}., it is given by a weak factorization of the codiagonal
$$
\nabla : 1+1 \xrightarrow{\quad  c_\varepsilon\quad} 2_\varepsilon
             \xrightarrow{\quad s_\varepsilon\quad} 1
$$
(see \cite{KR}). In general, for a space $K$, the cylinder object $C_K$ is given by a weak factorization
$$
\nabla : K+K \xrightarrow{\quad  c_K\quad} C_K
             \xrightarrow{\quad s_K\quad} K.
$$
Then, for morphisms $f,g:K\to L$ to admit a morphism $h: C_K \to L$ such that  
$$
\xymatrix@=3pc{
K+K \ar[rr]^{(f,g)} 
\ar[dr]_{c_K} && L\\
&C_K \ar[ur]_h
}
$$
commutes means precisely that  $f$ and $g$ are $\varepsilon$-homotopic. Indeed,  $K+K$ is obtained from $K$ by duplicating each $x\in K$ to $x'$ and $x''$ and putting $d(x',x'')=\infty$ while $d(x',x'')=\varepsilon$ in $C_K$.
   
We note that the sup-metric $d(f,g)=\sup\{d(fx,gx)\;|\;x\in K\}$ of $\Met(K,L))$ may be recovered from the $\eps$-homotopy relation, as
\[d(f,g)=\inf\{\eps\geq0\;|\; f\sim_\eps g\}.\]
}
\end{rem}

\begin{defi}\label{def2.2}
{
\em 

An {\em $\varepsilon$-pushout} of morphisms $f:A\to B,\; g:A\to C$ in a 
$\Met_{\infty}$-enriched category $\ck$ 
 is given by an $\varepsilon$-commutative square 
$$
\xymatrix@=4pc
{
B \ar[r]^{\overline{g}} & D \\
A\ar [u]^{f} \ar [r]_{g} & C \ar[u]_{\overline{f}}
}
$$
such that, for any $\varepsilon$-commutative square
$$
\xymatrix@=4pc{
B \ar[r]^{g'} & D' \\
A\ar [u]^{f} \ar [r]_{g} & C\,, \ar[u]_{f'}
}
$$
there is a unique morphism $t:D\to D'$ such that $t\overline{f}=f'$ and $t\overline{g}=g'$. An {\em $\eps$-coequalizer} of a pair of parallel morphisms is defined likewise. In the presence of coproducts we define the {\em $\eps$-colimit} of a diagram $D$ in $\ck$ as the $\eps$-coequalizer of the standard pair of morphisms between coproducts of the objects of the diagram that one uses to construct the (ordinary) colimit of $D$ from coproducts and coequalizers.
 }
 \end{defi}
 
 Up to isomorphism, $\eps$-colimits are uniquely determined; we denote the $\eps$-colimit of $D$ by $\colim_\eps D$. 0-colimits are simply colimits.
In case of a discrete diagram, the $\eps$-notion of colimit coincides with the ordinary one, for every $\eps\in[0,\infty]$:  $\eps$-coproducts are precisely coproducts.

\begin{lemma}\label{le2.3}
$\Met_{\infty}$ has $\varepsilon$-pushouts.
\end{lemma}
\begin{proof}
Since $\Met_{\infty}$ has pushouts, there is nothing to be shown in case $\eps=0$. For $\eps>0$, consider $f:A\to B$ and $g:A\to C$. In the coproduct $B+ C$ we have $d(fx,gx)=\infty$ for all $x\in A$. Changing all distances $d(fx,gx)$
to $\varepsilon$ gives a distance function that satisfies all axioms of a generalized metric but the triangle inequality. Such structures are called semimetrics and, following \cite{LR} 4.5(3), $\Met_{\infty}$ is reflective in the corresponding category $\SMet_{\infty}$: the reflector provides a semimetric space $(X,d)$ with the metric $\overline{d}$ given by

$$\overline{d}(x,z)={\rm inf}\{\sum_{i=0}^{n-1}d(y_i,y_{i+1})\,|\, n\geq1, y_i\in X, y_0=x, y_n=z\}.$$
It now suffices to take the reflection $D$ of the resulting semimetric space to $\Met_{\infty}$ to obtain an $\varepsilon$-pushout in $\Met_{\infty}$.
\end{proof}

\begin{coro}\label{cor2.4}
$\Met_{\infty}$  has $\eps$-colimits.
\end{coro}
\begin{proof}
An $\eps$-coequalizer
$$
\xymatrix@1{
A\ \ar@<0.6ex>[r]^{f}\
 \ar@<-0.6ex>[r]_g\ &\ B\ \ar [r]^h\ &\ D
}
$$
may be given by an $\eps$-pushout
$$
\xymatrix@C=3pc@R=3pc{
B \ar [r]^h  & D\\
A+ B \ar [u]^{(f, \id_B)} \ar [r]_{(g, \id_B)} &
B,\ar [u]_h
}
$$
and $\eps$-colimits are constructed 
with the help of coproducts and $\eps$-coequalizers. 
\end{proof}

\begin{propo}\label{prop2.5}
$\Ban$ has $\eps$-colimits.
\end{propo}
\begin{proof}
Let $\cs$ be a representative full subcategory of separable Banach spaces. Consider the functor
$$
E:\Ban\to\Met^{\cs^{\op}}
$$
given by $(EK)(A)=\Ban(K,A)$. The functor $E$ preserves limits and, following \ref{le1.1}, it preserves $\aleph_1$-directed colimits.
Since $E$ is a full embedding, it makes $\Ban$ a reflective full subcategory of $\Met^{\cs^{\op}}$. Following \ref{cor2.4}, $\Met^{\cs^{\op}}$
has $\eps$-colimits calculated pointwise. Given an $\eps$-diagram $D:\cd\to\Ban$, its $\eps$-colimit is given by a reflection of the $\eps$-colimit
of $ED$.
\end{proof}

\begin{rem}\label{re2.6} 
{
\em 
For a diagram 
$D$ in a $\Met_{\infty}$-enriched category $\ck$ with the needed ($\eps$-)colimits one has canonical morphisms
$$
\coprod_{i\in\cd}Di\simeq\colim_{\infty}D\to\colim_{\eps}D\to\colim_0D\simeq\colim D,
$$
with the morphisms 
$$
q_{\eps}:\colim_\eps D\to \colim D\quad(\eps>0)
$$ 
presenting $\colim D$ as a colimit of the chain $(\colim_{\eps}D\to\colim_{\delta}D)_{\eps\geq\delta>0}$.
}
\end{rem}

\section{$\varepsilon$-purity}

The notion of $\lambda$-pure morphism in a locally $\lambda$-presentable category as given in \cite{AR} allows for an obvious generalization in the case of a $\Met_{\infty}$-enriched category, as follows. The latter notion entails the former when one puts $\eps=0$.
\begin{defi}\label{def3.1}
{
\em
Let $\ck$ be a $\Met_{\infty}$-enriched category and $\lambda$ a regular cardinal.
We say that a morphism $f:K\to L$ is $\lambda$-$\varepsilon$-\textit{pure} if for any $\varepsilon$-commutative square
$$
\xymatrix@=4pc{
K \ar[r]^{f} & L \\
A\ar [u]^{u} \ar [r]_{g} & B \ar[u]_{v}
}
$$
with $A$ and $B$ $\lambda$-presentable in $\ck$ there exists $t:B\to K$ such that $tg\sim_\varepsilon u$.
}
\end{defi}

\begin{rem}\label{re3.2}
{
\em
(1) A composite of $\lambda$-$\varepsilon$-pure morphisms is $\lambda$-$\varepsilon$-pure.  

(2) If $f_2f_1$ is $\lambda$-$\varepsilon$-pure, then $f_1$ is $\lambda$-$\varepsilon$-pure.

(3) Every split monomorphism is $\lambda$-$\varepsilon$-pure, for any $\lambda$. (Indeed, when $pf=\id_A$, consider the $\varepsilon$-commutative square of
\ref{def3.1}. Then $fu\sim_\eps vg$ implies $u=pfu\sim_\eps pvg$.), 

(4) Every $\lambda$-$\varepsilon$-pure morphism is $\lambda'$-$\varepsilon$-pure, for all $\lambda'\leq\lambda$.

(5) The $\lambda$-$0$-pure morphisms are precisely the $\lambda$-pure morphisms (as defined in \cite{AR}).
}
\end{rem}

Before discussing $\lambda$-$\eps$-purity further, let us also consider some variations of the notion.

\begin{defi}\label{def3.3}
{
\em
Let $f:K\to L$ be a morphism in the $\Met_{\infty}$-enriched category $\ck$.
We say that a morphism $f:K\to L$ is \textit{weakly} (\textit{barely}) $\lambda$-$\varepsilon$-\textit{pure} if for every $\eps$-commutative (commutative) square as in \ref{def3.1}, with $A$ and $B$ $\lambda$-presentable in $\ck$, there exists $t:B\to K$ such that $tg\sim_{2\varepsilon} u$ $(tg\sim_\eps u$, respectively). We also say that $f$ is $\eps$-{\em split} if there is $p:L\to K$ in $\ck$ with $pf\sim_\eps {\rm id}_K$.  Finally, we say that $f$ is an $\eps$-monomorphism if $fg=fh$ implies $g\sim_\eps h$.
}
\end{defi}

The following easily verified statements all rely on the transitivity rule for $\eps$-homotopy:

\begin{lemma}\label{le3.4}
{\rm(1)} Every $\lambda$-$\eps$-pure morphism is weakly $\lambda$-$\eps$-pure and barely $\lambda$-$\eps$-pure.

{\rm(2)} Every weakly $\lambda$-$\eps$-pure morphism is barely $\lambda$-$2\varepsilon$-pure.

{\rm(3)} Every split monomorphism is $\eps$-split, and every $\eps$-split morphism is both, weakly and barely $\lambda$-$\eps$-pure, for any $\lambda$, and it is a $2\eps$-monomorphism.
\end{lemma}

\begin{rem}\label{re3.5}
{\em
(1) Note that an $\varepsilon$-split morphism does not need to be a monomorphism, not even an $\eps$-monomorphism: in $\Met_{\infty}$, for $0<\eps<\infty$, consider $\{a,b,c\}\to 1$ with $d(a,b)=d(b,c)=\eps$ and $d(a,c)=2\eps$.
}

{\rm (2)} For $\lambda$ uncountable, every barely $\lambda$-$\eps$-pure morphism in $\Met_{\infty}$ is $2\eps$-monomorphic. \em{ Indeed, 
for $f:K\to L$ barely $\lambda$-$\varepsilon$-pure, consider $a,b\in K$ with $fa=fb$. 
With $\delta=d(a,b)$ we exploit the commutative square 
$$
\xymatrix@=4pc{
K \ar[r]^{f} & L \\
2_\delta\ar [u]^{u} \ar [r]_{g} & 1, \ar[u]_{v}
}
$$
with $u$ mapping $2_\delta$ onto $\{a,b\}$. Since $2_\delta$ and $1$ are $\lambda$-presentable,
there is $t:1\to K$ such that $tg\sim_\varepsilon u$. With $c$ the image of $t$, this forces $d(a,b)\leq d(a,c)+d(b,c)\leq 2\varepsilon$.
}
\end{rem}
 
Let us also record to which extent $\lambda$-$\eps$-purity gets transported along $\eps$-homotopy:

\begin{lemma}\label{le3.6}
Let $f\sim_\eps f'$.

{\rm (1)} If $f$ is $\lambda$-$2\eps$-pure, then $f'$ is weakly $\lambda$-$\eps$-pure.

{\rm (2)} If $f$ is $\lambda$-$\eps$-pure, then $f'$ is barely $\lambda$-$\eps$-pure.
\end{lemma}

In conjunction with Remark \ref{re3.2} (2), Lemma \ref{le3.6} gives:

\begin{coro}\label{co3.7}
Let $gf\sim_\varepsilon h$. 
\begin{enumerate}
\item[(1)] If $h$ is $\lambda$-$2\varepsilon$-pure, then $f$ is weakly $\lambda$-$\varepsilon$-pure.
\item[(2)] If $h$ is $\lambda$-$\varepsilon$-pure, then $f$ is barely $\lambda$-$\varepsilon$-pure.
\end{enumerate}
\end{coro} 
 
We are now ready to prove an important stability property of $\lambda$-$\eps$-pure morphisms:

\begin{propo}\label{prop3.8}
Let $\ck$ be a $\Met_{\infty}$-enriched category and $\lambda$ be an uncountable regular cardinal. Then $\lambda$-$\varepsilon$-pure morphisms in $\ck$ are closed under $\lambda$-directed colimits in $\ck$ .
\end{propo}
\begin{proof}
Let $E:\ce\to\ck^\to$ be a $\lambda$-directed diagram in the morphism category of $\ck$ with $Ee:K_e\to L_e$ $\lambda$-$\varepsilon$-pure for all $e$ 
in $\ce$. For $f=\colim E:K\to L$ in $\ck$ with a colimit cocone $(k_e,l_e):Ee\to f$ we have that $k_e:K_e\to K$ and $l_e:L_e\to L$ are colimits
in $\ck$. Consider an $\eps$-commutative square
$$
\xymatrix@=4pc{
K \ar[r]^{f} & L \\
A\ar [u]^{u} \ar [r]_{g} & B \ar[u]_{v}
}
$$
with $A$ and $B$ $\lambda$-presentable in $\ck$. 

We will show that there are $e_0$ in $\ce$ and $u_{e_0}:A\to K_{e_0}$, $v_{e_0}:B\to L_{e_0}$ in $\ck$, such that $u=k_{e_0}u_{e_0}$, $v=l_{e_0}v_{e_0}$ and $(Ee_0)u_{e_0}\sim_\eps v_{e_0}g$. As $Ee_0$ is $\lambda$-$\eps$-pure, there is then  
$t:B\to K_{e_0}$ in $\ck$ with $tg\sim_\varepsilon u_{e_0}$. Hence, $u=k_{e_0}u_{e_0}\sim_\varepsilon k_{e_0}tg$, and the proof for $f$ to be $\lambda$-$\varepsilon$-pure will be complete. 

Indeed, since $A$ and $B$ are $\lambda$-presentable in $\ck_0$, first one finds $e$ in $\ce$ and $u_{e}:A\to K_{e}$ and $v_{e}:B\to L_{e}$ with $u=k_{e}u_{e}$ 
and $v=l_{e}v_{e}$.
Since $l_{e}(Ee)u_{e}=fk_{e}u_{e}=fu$ and $l_{e}v_{e}g=vg$, 
$$
d(l_{e}(Ee)u_{e},l_{e}v_{e}g)\leq\eps,
$$
follows. Since $A$ is $\lambda$-presentable in $\ck$, $\ck(A,l_e):\ck(A,L_e)\to\ck(A,L)$ is a colimit in $\Met_{\infty}$. 
By the construction of directed colimits in $\tilde{\Met}$, whereby 
$$
d(l_{e}(Ee)u_{e},l_{e}v_{e}g)={\rm inf}_{e'\geq e}d(l_{e,e'}(Ee)u_e,l_{e,e'}v_eg),
$$
with $(k_{e,e'},l_{e,e'}):Ee\to Ee'$ given by the diagram $E$, for all $n=1,2,\dots,$
there are then $e_n\geq e$ in $\ce$ with $d(l_{e,e_n}(Ee)u_{e},l_{e,e_n}v_{e}g)\leq\eps+\frac{1}{n}$. Finally, since $\lambda$ is uncountable, we can find $e_0\geq e_n$ for all $n$
and obtain $u=k_{e_0}u_{e_0}$, $v=l_{e_0}v_{e_0}$ and $(Ee_0)u_{e_0}\sim_\eps v_{e_0}g$.
\end{proof}

\begin{rem}\label{re3.9}
{
\em
As in Proposition \ref{prop3.8} one proves that the classes of weakly and barely $\lambda$-$\eps$-pure morphisms are both closed under $\lambda$-directed colimits in $\ck$, for $\lambda$ uncountable.
}
\end{rem}

\begin{coro}\label{cor3.10} 
Let $\lambda$ be a regular uncountable cardinal and $\ck$ be a $\Met_{\infty}$-enriched category with $\lambda$-directed colimits such that $\ck_0$
is locally $\lambda$-presentable. Then every $\lambda$-pure morphism is $\lambda$-$\varepsilon$-pure, for all $\eps\geq 0$.
\end{coro}
\begin{proof}
Since $\lambda$-pure morphisms are $\lambda$-directed colimits of split morphisms (see \cite{AR} 2.30), the result follows from Remark \ref{re3.2}(3)
and Proposition \ref{prop3.8}.
\end{proof}

We can finally give the following characterization of barely $\lambda$-$\eps$-pure morphisms in a large class of categories.
 
\begin{propo}\label{prop3.11}
Let $\lambda$ be an uncountable regular cardinal and $\ck$ be a $\tilde{\Met}$-enriched category with $\lambda$-directed colimits and $\eps$-pushouts 
such that $\ck_0$ is locally $\lambda$-presentable.  Then the following assertions are equivalent for a morphism $f$ in $\ck$:
\begin{enumerate}
\item[{\rm (i)}]  $f$ is barely $\lambda$-$\varepsilon$-pure;
\item[{\rm (ii)}] there are $g, h$ such that $gf\sim_\varepsilon h$, with $h$ being $\lambda$-pure;
\item[{\rm (iii)}] there are $g, h$ such that $gf\sim_\varepsilon h$, with $h$ being $\lambda$-$\eps$-pure.
\end{enumerate}
\end{propo}
\begin{proof} (i)$\Rightarrow$(ii):
Assume that $f:A\to B$ is barely $\lambda$-$\eps$-pure. We proceed as in the proof of \cite{AR} 2.30(ii) and express $f$ as a $\lambda$-directed colimit 
of morphisms $f_i:A_i\to B_i \;(i\in I)$, with $A_i$ and $B_i$ $\lambda$-presentable.  Since $f$ is barely $\lambda$-$\varepsilon$-pure, for every $i$ there is $t_i:B_i\to A$ such that $t_if_i\sim_\varepsilon u_i$.
Therefore, in the $\eps$-pushouts
$$
\xymatrix@=4pc{
A \ar[r]^{\overline{f}_i} & \overline{B_i} \\
A_i\ar [u]^{u_i} \ar [r]_{f_i} & B_i \ar[u]_{\overline{u}_i}
}
$$
every $\overline{f}_i$ is a split monomorphism. We get a $\lambda$-directed diagram $(\id_A,\overline{b}_{ij}):\overline{f}_i\to\overline{f}_j$.
Its colimit $\overline{f}:A\to\overline{B}$ is $\lambda$-pure (as a $\lambda$-directed colimit of of split monomorphisms), and it may be realized as the the $\varepsilon$-pushout 
$$
\xymatrix@=4pc{
A \ar[r]^{\overline{f}} & \overline{B} \\
A\ar [u]^{\id_A} \ar [r]_{f} & B, \ar[u]_{g}
}
$$ 
so that we have $gf\sim_\varepsilon\overline{f}$. 

(ii)$\Rightarrow$(iii): Corollary \ref{cor3.10}.

(iii)$\Rightarrow$(i): Corollary \ref{co3.7}(2).
\end{proof}
 
\section{$\varepsilon$-injectivity} 
\begin{defi}\label{def4.1}
{
\em
Let $\ck$ be a $\Met_{\infty}$-enriched category. Given a morphism $f:A\to B$, we say that an object $K$ is $\varepsilon$-\textit{injective} to $f$ if, 
for every $g:A\to K$, there exists $h:B\to K$ in $\ck$ with $hf\sim_\varepsilon g$. 
}
\end{defi}

\begin{rem}\label{re4.2}
{
\em
(1) $0$-injectivity coincides with the ordinary injectivity notion.

(2) If $K$ is $\varepsilon$-injective to $f$, then $K$ is also $\eps'$-injective to $f$, for all $\eps'\geq \eps$.

(3) $K$ is $\infty$-injective to $f:A\to B$ precisely when $\ck(B,K)=\emptyset$ only if $\ck(A,K)=\emptyset$.
}
\end{rem} 

For a class $\cf$ of morphisms in $\ck$, we denote by $$\Inj_\eps\cf$$ the class of objects $\eps$-injective to every $f\in\cf$. Trivially, following \ref{re4.2}(2), $\Inj_\eps\cf\subseteq\Inj_{\eps'}\cf$ whenever $\eps\leq\eps'$. A class of objects in $\ck$ is an $\varepsilon$-\textit{injectivity class} if, for some $\cf$, it is of the form $\Inj_\eps\cf$, and if $\cf$ is a set, then $\Inj_\eps\cf$ is called a \textit{small $\varepsilon$-injectivity class}. If the domains and codomains of morphisms in $\cf$ are all $\lambda$-presentable in $\ck$, $\Inj_\eps\cf$ is called a \textit{$\lambda$-$\eps$-injectivity class}. If $\ck_0$ is locally $\lambda$-presentable then every $\lambda$-$\eps$-injectivity class is a small $\eps$-injectivity class.
 
Compatibility of $\sim_\eps$ with the category composition immediately gives the expected closure properties of $\eps$-injectivity classes, as follows.

\begin{lemma}\label{le4.3}
Let $\cl$ be an $\varepsilon$-injectivity class in the $\Met_{\infty}$-enriched category $\ck$ with products. Then $\cl$ is closed under retracts, and $\cl$ is also closed under products in $\ck$.
\end{lemma}
\begin{proof}
Closure under retracts is obvious. For the product of of a family of $\eps$-injective objects $K_i$, since the canonical
$$
\ck(A,\prod_{i\in I}K_i)\to \prod_{i\in I}\ck(A,K_i),\quad g\mapsto (gp_i)_{i\in I},
$$
(with product projections $p_i$) is an isomorphism, one has
$$
\forall i\in I\;(p_i g\sim_\eps p_ig')\;\Longrightarrow\;g\sim_\eps g'
$$
whenever $g,g': A\to \prod_{i\in I}K_i$ in $\ck$, a property which is immediately seen to guarantee product stability of the $\eps$-injectivity class..
\end{proof}
 
\begin{lemma}\label{le4.4}
Let $\ck$ be a $\Met_{\infty}$-enriched category such that $\ck_0$ is locally presentable. Then every small $\varepsilon$-injectivity class in $\ck$ is closed under $\lambda$-directed colimits, for some regular cardinal $\lambda$. 
\end{lemma}
\begin{proof}
For any given set $\cf$ of morphisms in $\ck$ we can find $\lambda$ such that the domains of morphisms from $\cf$ are all $\lambda$-presentable. The proof that $\Inj_\eps\cf$ is closed
under $\lambda$-directed colimits is then straightforward. 
\end{proof}

We say that a class $\cl$ of objects is \textit{closed under (weakly) $\lambda$-$\eps$-pure morphisms} in $\ck$ if, for every (weakly) $\lambda$-$\eps$-pure morphism $K\to L$, with
$L$ in $\cl$ one has also $K$ in $\cl$.

\begin{lemma}\label{le4.5}
Let $\ck$ be a $\Met_{\infty}$-enriched category having all objects presentable. Then every small $\varepsilon$-injectivity class in $\ck$ is closed under 
$\lambda$-$\varepsilon$-pure morphisms, for some regular cardinal $\lambda$. 
\end{lemma}
\begin{proof}
Take $\lambda$ such that the domains and the codomains of morphisms in $\cf$ are all $\lambda$-presentable in $\ck$.
Let $p:K\to L$ be $\lambda$-$\eps$-pure with $L$ in $\Inj_\eps\cf$, and consider $f: A\to B$ in $\cf$ and any $g:A\to K$. 
$\eps$-injectivity of $L$ gives $h:B\to L$ such that $gf\sim_\eps pg$, and then $\lambda$-$\eps$-purity of $p$ gives a morphism $t: B\to K$ with $tf\sim_\eps g$.
\end{proof}

We now have the tools enabling us to state:

\begin{propo}\label{prop4.6}
Let $\ck$ be a $\Met_{\infty}$-enriched category with products, such that all objects in $\ck$ are presentable and the ordinary category $\ck_0$ is locally presentable. Then every small 
$\varepsilon$-injectivity class in $\ck$ is a small injectivity class.
\end{propo}
\begin{proof}
According to Theorem 4.8 in \cite{AR}, it suffices to show that a small $\eps$-injectivity class is accessible and accessibly embedded into the ambient locally presentable category, as well as closed under products. While the latter condition is satisfied by Lemma \ref{le4.3}, the former two conditions are guaranteed by Lemmas \ref{le4.4} and \ref{le4.5}, in conjunction with Corollary 2.36 in \cite{AR}.
\end{proof}

\begin{rem}\label{re4.7}
{
\em
By Theorem 2.2 in \cite{RAB}, in a locally $\lambda$-presentable category, $\lambda$-injectivity (= $\lambda$-$0$-injectivity) classes are characterized by closure under products, $\lambda$-directed colimits and $\lambda$-pure subobjects. Consequently, every $\lambda$-$\varepsilon$-injectivity class in a category satisfying the hypotheses of Proposition \ref{prop4.6} is a $\lambda$-injectivity class. 
}
\end{rem}

\begin{theo}\label{th4.8}
Let $\lambda$ be an uncountable regular cardinal and $\ck$ a $\Met_{\infty}$-enriched category with $\lambda$-directed colimits, such that $\ck_0$ is
locally $\lambda$-presentable and any $\lambda$-presentable object in $\ck_0$ is $\lambda$-presentable in $\ck$. Then every class 
$\cl$ of objects in $\ck$ closed under products, $\lambda$-directed colimits and weakly $\lambda$-$\eps$-pure morphisms is a $\lambda$-$\eps$-injectivity class and, in particular, a small injectivity class.
\end{theo}
\begin{proof}
Let $\cl$ be closed under products, $\lambda$-directed colimits and weakly $\lambda$-$\eps$-pure morphisms. We will follow the proof of \cite{RAB} 2.2.
According to \ref{cor3.10} and \cite{AR} 2.36 and 4.8, $\cl$ is weakly reflective. This means that every $K$ in $\ck$ comes with a morphism
$r_K:K\to K^\ast$, $K^\ast\in\cl$, such that every object of $\cl$ is injective to $r_K$. Let $\cf$ consist of all morphisms $f:A\to B$ such that $A$ 
and $B$ are $\lambda$-presentable and every object of $\cl$ is $\eps$-injective to $f$. By the definition of $\cf$ we have $\cl\subseteq\Inj_\eps\cf$, and the converse inclusion
$\Inj_\eps\cf\subseteq\cl$ will follow from the closure of $\cl$ under weakly $\lambda$-$\eps$-pure morphisms once we have shown
that, for $K\in\Inj_\eps\cf$, any weak reflection of $K$ into $\cl$ is weakly $\lambda$-$\eps$-pure.
 
Thus, given $K\in\Inj_\eps\cf$ and a weak reflection $r:K\to K^*$ in $\cl$, we are to prove that in any $\eps$-commutative square
\vskip 2 mm
$$
\xymatrix@C=3pc@R=3pc{
A \ar[r]^{h} \ar [d]_u &
B \ar [d]^{v}\\
K\ar[r]_{r} & K^*
}
$$
\vskip 2 mm\noindent
with $A$ and $B$ $\lambda$-presentable the morphism $u$ $2\eps$-factors through $h$. We will say that $(u,v) :h\to r$ is an $\eps$-morphism in this situation.

\vskip 2mm
\noindent {\it Claim}: There is a factorization $u=u_2\cdot u_1$ and an $\eps$-morphism $(u_1,v_1):h\to \bar r$ where $\bar r:\bar K\to\bar K^*$ 
is a weak reflection into $\cl$ of a $\lambda$-presentable $\bar K$.

\vskip2mm
\noindent {\it Proof of claim}. Consider all $\eps$-morphisms $(u_1,v_1):h\to\bar r$ where $\bar r:\bar K\to \bar K^*$ is a weak reflection of
$\bar K$ in $\cl$ and $u=u_2\cdot u_1$ for some $u_2$. Since $(u,v):h\to r$ is such an $\eps$-morphism, we can take the smallest $\alpha$ such that $\bar K$ is $\alpha$-presentable. We are to prove $\alpha\leq\lambda$. Assuming $\alpha >\lambda$ we will obtain a contradiction. As in \cite{RAB}, we express
$\bar K$ as a colimit of a smooth chain $k_{ij}:K_i\to K_j$ $(i\leq j<\alpha)$ of objects $K_i$ of presentability less that $\alpha$. This provides
weak reflections $r_i:K_i\to K_i^\ast$ into $\cl$ such that their colimit $r_\alpha:\bar K\to K_\alpha^\ast$ factorizes through $\bar r$, i.e.,
$r_\alpha=s\bar r$ for some $s:\bar K^\ast\to K_\alpha^\ast$. Since $\bar ru_1\sim_\eps v_1h$, we have $r_\alpha u_1=s\bar ru_1\sim_\eps sv_1h$, so that $(u_1,sv_1):h\to r_\alpha$ is an $\eps$-morphism. In the same way as in the proof of \ref{prop3.8}, this $\eps$-morphism $\eps$-factors
through some $r_i$, $i<\alpha$. This means that there is an $\eps$-morphism $h\to r_i$, which contradicts the minimality of $\alpha$ and proves the claim.

We are ready to prove that $u$ $2\eps$-factors through $h$. Let us consider a factorization $u=u_2\cdot u_1$ and a morphism $(u_1,v_1):h\to \bar r$ 
as in the above claim. Let us express $\bar K^*$ as a $\lambda$-directed colimit of $\lambda$-presentable objects $Q_t$, $t\in T$, with  a
colimit cocone $q_t:Q_t\to\bar K^*$. Since both $\bar K$ and $B$ are $\lambda$-presentable, the morphisms $\bar r$ and $v_1$ both factor through $q_{t_0}$ for some $t_0\in T$. Since $A$ is $\lambda$-presentable, there then exists $t_1\geq t_0$ in $T$ with an $\eps$-commutative diagram, as follows:
\vskip 2 mm
$$
\xymatrix@C=3pc@R=3pc{
A \ar [dd]_{u_1}
  \ar [rr]^{h} &&
B \ar[dd]^{v_1}
  \ar [dl]_{\tilde v_1}  \\
& Q_{t_1}\ar[dr]^{q_{t_1}}  &\\
\bar K \ar[ur]^{\tilde r}
  \ar[rr]_{\bar r}  && \bar K^*
}
$$
\vskip 2 mm\noindent
Since all objects of $\cl$ are injective to $\bar r$, they are also injective to $\tilde r$; moreover, $\bar K$ and $Q_{t_1}$ are
both $\lambda$-presentable. Thus $\tilde r\in\cf$. This implies that $K$ is $\eps$-injective to $\tilde r$. Choosing $d:Q_{t_1}\to K$ with 
$u_2\sim_\eps d\tilde r$ we obtain
$$
u=u_2u_1\sim_\eps d\tilde ru_1\sim_\eps d\tilde v_1h.
$$
Hence, $r$ is weakly $\lambda$-$\eps$-pure, and thus $K$ lies in $\cl$.
\end{proof}

\begin{rem}\label{re4.9}
{
\em
(1) Let 
$$
\xymatrix@=4pc{
B \ar[r]^{\overline{g}} & D \\
A\ar [u]^{f} \ar [r]_{g} & C \ar[u]_{\overline{f}}
}
$$
be an $\eps$-pushout and $K$ be $\eps$-injective to $f$. Then $K$ is injective to $\overline{f}$.
Indeed, considering $u:C\to K$ we obtain $v:B\to K$ such that $vf\sim_\eps u$. Thus there is $w:D\to K$ such that $w\overline{f}=u$.

(2) Let
$$
\xymatrix@=4pc{
A \ar[r]^{\overline{f}} & \overline{B} \\
A\ar [u]^{\id_A} \ar [r]_{f} & B \ar[u]_{p}
}
$$ 
be an $\eps$-pushout as in the proof of \ref{prop3.11} (which corresponds to the mapping cylinder in homotopy theory). Then an object $K$ is $\eps$-injective to $f$ if and only if it injective to $\overline{f}$. Indeed,
the ``if"-part of this statement follows from (1), and the converse is evident.

(3) If $\lambda$ is an uncountable regular cardinal and $A,B$ are $\lambda$-presentable, then $\overline{B}$ in (2) is $\lambda$-presentable.
The verification is analogous to that in the proof of \ref{prop3.8}.

This yields, under the presence of $\eps$-pushouts, a direct proof of \ref{th4.8}.
}
\end{rem} 

\begin{pb}\label{pb4.10}
{
\em
Let $\lambda$ be an uncountable regular cardinal and $\ck$ a $\Met_{\infty}$-enriched category with $\lambda$-directed colimits, such that $\ck_0$ is
locally $\lambda$-presentable and any $\lambda$-presentable object in $\ck_0$ is $\lambda$-presentable in $\ck$. Are $\lambda$-$\eps$-injectivity classes 
in $\ck$ precisely classes closed under products, $\lambda$-directed colimits and $\lambda$-$\eps$-pure morphisms?
}
\end{pb}

\section{Approximate injectivity}
The following definition is motivated by \cite{L} 3.2.
\begin{defi}\label{def5.1}
{
\em
Let $\ck$ be a $\Met_{\infty}$-enriched category.
We say that an object $K$ is \textit{approximately injective} to $f:A\to B$ in $\ck$ if it is $\varepsilon$-injective
to $f$ for every $\varepsilon>0$.
}
\end{defi}
The class of objects in $\ck$ approximately injective to a class $\cf$ of morphisms in $\ck$ will be denoted $\Inj_{\ap}\cf$. If $\cf$ is a set, then $\Inj_{\ap}\cf$ is called
an \textit{approximate small injectivity class}. If the domains and the codomains of all morphisms in $\cf$ are $\lambda$-presentable in $\ck$,
$\Inj_{\ap}\cf$ is called an \textit{approximate $\lambda$-injectivity class}. If $\ck_0$ is locally presentable, then any approximate $\lambda$-injectivity class is an approximate small injectivity class.
 
\begin{defi}\label{def5.2}
{
\em 
A morphism in $\ck$ is \textit{(weakly, barely)} $\lambda$-$\ap$-\textit{pure} if it is (weakly, barely) $\lambda$-$\eps$-pure for every $\eps>0$.
}
\end{defi}

\begin{rem}\label{re5.3}
{
\em
(1) A composite of $\lambda$-$\ap$-pure morphisms is $\lambda$-$\ap$-pure.

(2) If the composite morphism $f_2f_1$ is $\lambda$-$\ap$-pure, then $f_1$ is also $\lambda$-$\ap$-pure.

(3) Let $\lambda$ be an uncountable regular cardinal and $\ck$ be a $\Met_{\infty}$-enriched category with $\lambda$-directed colimits and $\eps$-pushouts, 
such that $\ck_0$ is locally $\lambda$-presentable. Then every barely $\lambda$-$\ap$-pure morphism is a monomorphism. Indeed,
considering the $\eps$-pushouts
$$
\xymatrix@=4pc{
A \ar[r]^{f_\eps} & B_\eps \\
A\ar [u]^{\id_A} \ar [r]_{f} & B, \ar[u]_{g}
}
$$ 
and applying the characterization \ref{prop3.11}(ii), since $\lambda$-pure morphisms in an accessible category are monomorphisms, we see that every $f_\eps$ is monic. Consequently, as a directed colimit of these morphisms, also $f$ is a monomorphism.

(4) It follows easily from Lemma \ref{le3.4}(2) that every weakly $\lambda$-ap-pure morphism  is barely $\lambda$-ap-pure and, hence, a monomorphism, by (3). Consequently, rather than referring to its closure under weakly $\lambda$-$\ap$-pure morphisms 
we may say that a class $\cl$ of objects be \textit{closed under weakly $\lambda$-$\ap$-pure subobjects}.
}
\end{rem}

\begin{lemma}\label{le5.4}
Let $\lambda$ be an uncountable regular cardinal and $\ck$ a $\tilde{\Met}$-enriched category with $\lambda$-directed colimits, such that $\ck_0$
is locally $\lambda$-presentable. Then every $\lambda$-pure morphism is 
$\lambda$-$\ap$-pure.
\end{lemma}
\begin{proof}
The statement follows from Corollary \ref{cor3.10}.
\end{proof}

\begin{theo}\label{th5.5}
Let $\lambda$ be an uncountable regular cardinal and $\ck$ a $\Met_{\infty}$-enriched category with $\lambda$-directed colimits, such that $\ck_0$ is locally $\lambda$-presentable and every $\lambda$-presentable object in $\ck_0$ is $\lambda$-presentable in $\ck$. Then the approximate $\lambda$-injectivity classes in $\ck$ are precisely the full subcategories closed under products, $\lambda$-directed colimits and weakly $\lambda$-$\ap$-pure morphisms.
\end{theo}
\begin{proof}

Since
$$
\Inj_{\ap}\cf=\bigcap\limits_{\eps>0}\Inj_\eps\cf,
$$
every approximate $\lambda$-injectivity class is closed under products and $\lambda$-directed colimits (see \ref{le4.3} and \ref{le4.4}).
We will show that $\Inj_{\ap}\cf$ is closed under weakly $\lambda$-$\ap$-pure subobjects. Let $p: K\to L$ be $\lambda$-$\ap$-pure and $L$
belong to $\Inj_{\ap}\cf$. Consider $f:A\to B$ in $\cf$, $\eps>0$ and $u:A\to K$. There is $v:B\to L$ such that $pu\sim_{\frac{\eps}{2}}vf$. Since
$p$ is weakly $\lambda$-$\frac{\eps}{2}$-pure, there exists $t:B\to K$ with $tf\sim_\eps u$. Thus $K$ is $\eps$-injective to $f$.
  
Let $\cl$ be closed under products, $\lambda$-directed colimits and weakly $\lambda$-$\ap$-pure subobjects. We will proceed in the same way as in \ref{th4.8}. 
Let $\cf$ consist of all morphisms $f:A\to B$ such that $A$ and $B$ are $\lambda$-presentable and every object of $\cl$ is approximately injective to $f$. We have $\cl\subseteq\Inj_{\ap}\cf$, and the converse inclusion will follow from the fact that every weak reflection of $K\in\Inj_{\ap}\cf$ into $\cl$ is weakly $\lambda$-$\ap$-pure. Since $K\in\Inj_{\frac{\eps}{2}}\cf$, a weak reflection $r:K\to K^\ast$ is weakly 
$\lambda$-$\eps$-pure. Hence $r$ is $\lambda$-$\ap$-pure.
\end{proof}

In the presence of $\eps$-pushouts, we can speak about the closure under weakly $\lambda$-$\ap$-pure subobjects (see \ref{re5.3}(4)).

\begin{coro}\label{cor5.6}
Under the hypotheses of Theorem \ref{th5.5},  every approximate $\lambda$-injectivity class in $\ck$ is a $\lambda$-injectivity class.
\end{coro}
\begin{proof}
The statement follows from \ref{le5.4}, \ref{th5.5} and \cite{RAB} 2.2.
\end{proof}

\begin{rem}\label{re5.7}
{
\em
Continuing to work under the hypotheses of Theorem \ref{th5.5}, we let $\cl$ be an approximate 
$\lambda$-injectivity class in $\ck$. Then, following \ref{cor5.6} and \cite{AR} 4.8, $\cl$ is weakly reflective. We claim  that every object $K$ that is approximately injective to its weak reflection $r:K\to K^\ast$ must lie in $\cl$. Indeed, $r$ $\eps$-splits for every $\eps>0$ and then, by \ref{le3.4}, must be weakly $\lambda$-$\ap$-pure, for any 
$\lambda$. Hence, $K\in\cl$ follows.
}
\end{rem}

Recall that {\em Vop\v  enka's Principle} is a large-cardinal principle which guarantees that injectivity classes in a locally presentable category are characterized by their closure under products and split subobjects (see \cite{AR} 6.26). We are now ready to conclude the validity of an ``ap-version" of this theorem. To state it, we say that $f$ is $\ap$-\textit{split} if it is $\eps$-split for every $\eps>0$.

\begin{theo}\label{th5.8}
Under Vop\v  enka's Principle,
the following conditions are equivalent for a full subcategory $\cl$ of a category $\ck$ satisfying the hypotheses of Theorem \ref{th5.5}:
\begin{enumerate}
\item $\cl$ is closed under products and $\ap$-split subobjects,
\item $\cl$ is an $\ap$-injectivity class,
\item $\cl$ is weakly reflective and closed under $ap$-split subobjects.
\end{enumerate}
\end{theo}
\begin{proof}
$(3)\Rightarrow(2)$: By Remark \ref{re5.7}, $\cl$ is an $\ap$-injectivity class w.r.t. weak reflections of $\ck$-objects to $\cl$.

\noindent
$(2)\Rightarrow(1)$ follows from \ref{le3.4} and \ref{th5.5}.

\noindent
$(1)\Rightarrow(3)$ follows from \cite{AR} 6.26, since closure under ap-split subobjects trivially entails  closure under split subobjects.
\end{proof}

\section{The countable case}
Regular monomorphisms in $\Met_{\infty}$ are isometries, and these are stable under pushout. Every finite generalized metric space $A$ is $\aleph_0$-\textit{generated}, in the sense that $\Met_{\infty}(A,-):\Met_{\infty}\to\Set$ preserves directed colimits of isometries. 

A generalized metric space $K$ is $\aleph_0$-\textit{saturated} if, for any isometry $f:A\to B$ between finite generalized metric spaces and any isometry $g:A\to K$, there is an isometry $h:B\to K$ with $hf=g$. This means that $K$ is injective to morphisms
between finite (i.e., $\aleph_0$-generated) objects in the category of generalized metric spaces and isometries.

A generalized metric space is called \textit{rational} if all of its distances are either rational or $\infty$. By an $\aleph_0$-saturated generalized rational metric space we mean a space which is injective to morphisms between finite spaces in the category of generalized rational metric spaces and isometries.

\begin{theo}\label{th6.1}
There is a countable $\aleph_0$-saturated generalized rational metric space.
\end{theo}
\begin{proof}
Isometries in $\Met_{\infty}$ are stable under pushout; moreover, if the given spaces are rational, so is the pushout. Up to isomorphism, there 
are only countably many finite generalized rational metric spaces. Let $\cs$ be the set of all isometries between them; $\cs$ is countable again. We express $\cs$ as a union of a countable chain of finite subsets $\cs_n$ and will construct a countable chain 
$(k_{ij}:K_i\to K_j)_{i<j<\omega}$ of finite generalized rational metric spaces and isometries, as follows. Let $K_0=\emptyset$. Having $(k_{ij}:K_i\to K_j)_{i<j\leq n}$, we take the diagram consisting of all spans $(u,h)$ where $h\in\cs_n$. Then 
$k_{n,n+1}:K_n\to K_{n+1}$ is given by the corresponding multiple pushout, i.e., by the pushout
$$
\xymatrix@=4pc{
K_n \ar[r]^{k_{n,n+1}} & K_{n+1} \\
\coprod_h X_h\ar [u]^{<u>} \ar [r]_{\coprod h} & \coprod_h Y_h \ar[u]_{}
}
$$
with $h:X_h\to Y_h$ running through $\cs_n$.
As the colimit of a chain of isometrically embedded rational generalized metric spaces, also $K=\colim K_n$ is rational. We claim that $K$ is $\aleph_0$-saturated. Indeed, consider $h:X\to Y$ in $\cs$ and $u:X\to K$. Since $X$ is $\aleph_0$-small,
there is $u':X\to K_n$ such that $k_nu'=u$; here $k_n:K_n\to K$ is a colimit injection. Without any loss of generality, we may assume $h\in\cs_n$. Thus $k_{n,n+1}u'=vh$ for some $v:Y\to K_{n,n+1}$. Hence $u=k_nu'=vh$, as desired.
\end{proof}

A countable $\aleph_0$-saturated generalized rational metric space $U_0$ is, in fact, uniquely determined, up to isomorphism: see, e.g., \cite{R} Theorem 2; it is the Fra\" iss\' e limit of finite generalized rational metric spaces (see \cite{K1}). Its completion $U$ is an 
$\aleph_0$-saturated complete separable metric space, called \textit{Urysohn space} in the literature: see \cite{H} or \cite{K1} for a proof of its so-called universality and homogeneity, from which one easily concludes its $\aleph_0$-saturatedness in $\bf{Met}$. 

In order for us to establish a corresponding result in $\Ban$, we introduce the needed definitions more generally at the level of $\Met_{\infty}$-enriched categories.

\begin{defi} \label{def6.2}
{
\em
Let $\ck$ be a $\Met_{\infty}$-enriched category. A morphism $f:A\to B$ is called an \textit{isometry} if, for every $\eps\geq 0$
and all $u,v:C\to A$, one has
$$
fu\sim_\eps fv\Rightarrow u\sim_\eps v.
$$
}
\end{defi}

\begin{exam}\label{ex6.3}
{
\em
Both in $\Met$ and $\Ban$, isometries have their usual meaning. In $\Met$, it suffices to test them on $u,v:1\to A$, and in $\Ban$ on $u,v:l_1(1)\to A$.
}
\end{exam}

\begin{defi}\label{def6.4}
{
\em
Let $\ck$ be a $\Met_{\infty}$-enriched category. An object $A$ in $\ck$ is $\lambda$-$\eps$-\textit{generated} if, for any $\lambda$-directed
diagram of isometries $(k_{ij}:K_i\to K_j)_{i\leq j\in I}$ with colimit cocone $k_i:K_i\to K$ and every morphism $f:A\to K$, there is $i\in I$ such that
\begin{enumerate}
\item $f$ $\eps$-factorizes through $k_i$, {\em i.e.}, $f\sim_\eps k_ig$ for some $g:A\to K_i$,
\item the $\eps$-factorization is $\eps$-essentially unique, in the sense that, if $f\sim_\eps k_ig$ and $f\sim_\eps k_ig'$, then
$k_{ij}g\sim_\eps k_{ij}g'$ for some $j\geq i$.
\end{enumerate}

We say that $A$ is $\lambda$-$\ap$-\textit{generated} if it is $\lambda$-$\eps$-
generated for every $\eps>0$.
}
\end{defi}

\begin{exam}\label{ex6.5}
{
\em
(1) In $\Ban$, every finite-dimensional space $A$ is $\aleph_0$-$\ap$-generated. 

(2) Since every Banach space is a directed colimit of finite-dimensional Banach spaces and isometries, every $\aleph_0$-$\eps$-generated Banach space
admits an $\eps$-split morphism to a finite-dimensional Banach space, for any $\eps>0$.

(3) More generally, for every isometry $f:X\to Y$ between $\aleph_0$-$\ap$-generated Banach spaces and every $\eps>0$, there is a commutative square
$$
\xymatrix@=4pc{
X \ar[r]^{f} & Y \\
A\ar [u]^{u} \ar [r]_{g} & B \ar[u]_{v}
}
$$
in $\Ban$ with an isometry $g$ between finite-dimensional Banach spaces $A, B$, as well as morphisms
$s:X\to A$, $t:Y\to B$, such that $us\sim_\eps\id_X$, $vt\sim_\eps\id_Y$ and $gs\sim_\eps tf$.
}
\end{exam}

\begin{defi}\label{def6.6}
{
\em
Let $\ck$ be a $\Met_{\infty}$-enriched category. We say that an object $K$ is $\aleph_0$-$\ap$-\textit{saturated} if it is approximately injective to morphisms between $\ap$-$\aleph_0$-generated objects in the category of $\ck$-objects and isometries.
}
\end{defi}

\begin{theo}\label{th6.5}
There is a separable $\aleph_0$-$\ap$-saturated Banach space.
\end{theo}
\begin{proof}
By \ref{ex6.5}(2), a Banach space is $\aleph_0$-ap-saturated if, and only if, it is approximately injective to isometries between finite-dimensional
Banach spaces. Now we proceed in the same way as in \ref{th6.1}. Isometries are pushout stable in $\Ban$ (see \cite{ASCGM}, 2.1) and, following \cite{GK} 2.7 and \cite{K2}, a Banach space is approximately injective to finite-dimensional Banach spaces if, and only if, it is approximately injective to rational
isometries between finite-dimensional rational Banach spaces. (For the meaning of ``rational" in the Banach space context, see \cite{GK}.) Up to isomorphism, there are only countably many such isometries. We will show that $K$, constructed analogously to the construction in \ref{th6.1}
from the countable set $\cs$ of relevant isometries as a colimit of separable spaces $K_n$ (with $K_0$ the null space), is 
$\aleph_0$-$\ap$-saturated. For that, consider $h:X\to Y$ in $\cs$ and $u:X\to K$, and let $\eps>0$. Since $X$ is $\aleph_0$-$\ap$-generated, there is $u':X\to K_n$ 
such that $k_nu'\sim_\eps u$. Without loss of generality we may assume that $h\in\cs_n$. Consequently, $k_{n,n+1}u'=vh$ for some $v:Y\to K_{n,n+1}$ and, hence, 
$$
u\sim_\eps k_nu'=k_{n+1}k_{n,n+1}u'=k_{n+1}vh,
$$
as desired.
\end{proof}

A separable $\aleph_0$-ap-saturated Banach space is in fact uniquely determined, up to isomorphism; it coincides with the {\em Gurarii space} (see \cite{K2}).

\end{document}